\documentclass[11pt,twoside]{amsart}

\usepackage{amsxtra}
\usepackage{color}
\usepackage{amsopn}
\usepackage{amsmath,amsthm,amssymb,arydshln}
\usepackage{mathrsfs,mathtools}
\usepackage{stmaryrd}
\usepackage{hyperref}
\usepackage{enumerate}
\usepackage{xcolor}
\usepackage{pifont}
\usepackage{adjustbox}
\usepackage{tikz}

\newtheorem{theorem}{Theorem}[section]
\newtheorem{corollary}[theorem]{Corollary}
\newtheorem{proposition}[theorem]{Proposition}
\newtheorem{lemma}[theorem]{Lemma}
\newtheorem*{theorem*}{Theorem}
\theoremstyle{definition}
\newtheorem{definition}[theorem]{Definition}

\newtheorem{example}[theorem]{Example}
\newtheorem{remark}[theorem]{Remark}

\def\sideremark#1{\ifvmode\leavevmode\fi\vadjust{
		\vbox to0pt{\hbox to 0pt{\hskip\hsize\hskip1em
				\vbox{\hsize3cm\tiny\raggedright\pretolerance10000
					\noindent #1\hfill}\hss}\vbox to8pt{\vfil}\vss}}}

\newcommand{\bC}{\mathbb{C}}

\newcommand{\bR}{\mathbb{R}}
\newcommand{\bZ}{\mathbb{Z}}

\newcommand{\beq}{\begin{equation}}
\newcommand{\eeq}{\end{equation}}
\renewcommand{\a}{\alpha}
\renewcommand{\b}{\beta}

\renewcommand{\d}{\delta}
\newcommand{\e}{\epsilon}

\renewcommand{\l}{\lambda}

\renewcommand{\o}{\omega}

\newcommand{\s}{\sigma}

\newcommand{\U}{{\mathrm U}}
\newcommand{\SU}{{\mathrm{SU}}}

\newcommand{\SO}{{\mathrm {SO}}}

\DeclareMathOperator\Ad{Ad}
\DeclareMathOperator\ad{ad}

\newcommand{\gge}{\mathfrak{e}}
\newcommand{\gf}{\mathfrak{f}}
\renewcommand{\gg}{\mathfrak{g}}
\newcommand{\gh}{\mathfrak{h}}
\newcommand{\gk}{\mathfrak{k}}
\newcommand{\gl}{\mathfrak{l}}
\newcommand{\gm}{\mathfrak{m}}
\newcommand{\gn}{\mathfrak{n}}

\newcommand{\gp}{\mathfrak{p}}
\newcommand{\gq}{\mathfrak{q}}

\newcommand{\gs}{\mathfrak{s}}
\newcommand{\gt}{\mathfrak{t}}
\newcommand{\gu}{\mathfrak{u}}

\newcommand{\gz}{\mathfrak{z}}

\newcommand{\so}{\mathfrak{so}}
\newcommand{\su}{\mathfrak{su}}

\newcommand{\gsp}{\mathfrak{sp}}
\newcommand{\gsl}{\mathfrak{sl}}

\newcommand\es[1]{\epsilon_{\sigma({#1})}}

\numberwithin{equation}{section}

\begin{document}

\title[Homogeneous almost K\"ahler manifolds and the Chern-Einstein equation ] {Homogeneous almost K\"ahler manifolds and the Chern-Einstein equation}

\author{Dmitri V. Alekseevsky}
\address{Institute for Information Transmission Problem RAS, B. Karetnuj per., 19, 127951, Moscow,RUSSIA {\it and } Faculty of Science, University of Hradec Kralove, Rokitanskeho 62, Hradec Kralove 50003, CZECH REPUBLIC.}
\email{dalekseevsky@iitp.ru}
\author{ and Fabio Podest\`a }
\address{Dipartimento di Matematica e Informatica "Ulisse Dini", Universit\`a di Firenze, V.le Morgagni 67/A, 50100 Firenze, ITALY}
\email{fabio.podesta@unifi.it}



\begin{abstract} Given a non compact semisimple Lie group $G$ we describe all homogeneous spaces $G/L$ carrying an invariant almost K\"ahler structure $(\o,J)$. When $L$ is abelian and $G$ is of classical type, we classify all such spaces which are Chern-Einstein, i.e. which satisfy $\rho = \l\o$ for some $\l\in\mathbb R$, where $\rho$ is the Ricci form associated to the Chern connection.\end{abstract}

\maketitle
\section{Introduction}
Given an almost K\"ahler manifold $(M,g,J)$, that is  an Hermitian manifold  with  closed  K\"ahler  form
$\omega = g \circ J$, the  Ricci form $\rho$ of the associated Chern connection $D$ is a closed $2$-form which represents the cohomology class $2\pi c_1(M,J)$ in $H^2(M,\bR)$.  The Chern-Einstein equation $\rho = \lambda \o$ for some $\l\in\bR$ gives a very natural generalization of the Einstein condition. This equation has been considered in \cite{AD} and more recently in \cite{DV}. More generally, we can consider a symplectic manifold $(M,\o)$ and study the existence of a compatible almost complex structure $J$ so that the Chern- Einstein equation $\rho = \l\o$ is satisfied. In \cite{DV} several examples of non-compact homogeneous examples of Chern-Einstein almost K\"ahler manifolds are given and some structure theorems are proved. \par
In this work we focus on  non-compact symplectic manifolds $(M,\o)$ which admit a (non-compact) semisimple Lie group $G$ of transitive symplectomorphisms with compact isotropy subgroup. Our first result is stated in Theorem \ref{Thm1} and it shows that there exists a unique $G$-homogeneous almost complex structure $J$ which is compatible with $\o$. This result has to be contrasted to the well-known case when $G$ is compact and the homogeneous almost complex structure is necessarily integrable (see e.g. \cite{WG}). We then study the Chern-Einstein equation in the homogeneous setting and our main  result  is  the explicit  classification  of invariant Chern-Einstein  
metrics  on regular  adjoint  orbits of the  form  $G/L$                     where $G$ is a classic non-compact simple
Lie  group   and  $L$  is a  compact abelian subgroup (a maximal  torus  of   $G$). We  prove  that  such a metric  exists precisely for $\gg =   \gsl(2,\bR)$  and  $\su(p+1,p)$ ($p\geq 1$) where $\gg$ is the Lie algebra of $G$. In particular,  we prove  the  following

\begin{theorem}\label{MT} Let $(M=G/L,\o)$ be a homogeneous symplectic manifold of  a  non compact  semisimple  Lie  group  and $L$ compact. Then there exists a unique invariant almost complex structure $J$ compatible with $\o$ so that $(M,\o,J)$ is almost K\"ahler. \par
If  the  group  $G$   is  simple non compact of classical  type   and $L$ is abelian, then $(M,\o,J)$ is Chern Einstein, say $\rho =\l\omega$, if and only if one of the following occurs:
	\begin{enumerate}
		\item[i)] $\l<0$ and $\gg = \gsl(2,\mathbb R)$
		\item[ii)] $\l = 0$ and  $\gg = \su(p+1,p)$,\ $p\geq 1$.
		\end{enumerate}
\end{theorem}
While the full classification of the Chern-Einstein non compact homogeneous almost K\"ahler manifolds remains out of reach by this time, we can give some simple examples when the isotropy has one dimensional center (see section~\ref{Ex}). We also remark that the case when $G$ is semisimple can be easily deduced from the main result, as the manifold will split as an almost K\"ahler product of $G_i$- homogeneous manifolds where $G_i$ are the simple factors of $G$. \par
In section 2 we describe the homogeneous almost K\"ahler non-compact manifolds, which are acted on transitively by a semisimple Lie group. In section 3 we describe the Chern connection in the homogeneous setting and give a formula for the Ricci form $\rho$, which is analogue to the standard formula in the compact case (see e.g. \cite{AP}, \cite{BFR}). In the last section we describe the general strategy to prove the classification claimed in our main theorem and we prove it going through each simple Lie algebra of classical type. We also indicate how to produce homogeneous Chern-Einstein manifolds with non abelian compact isotropy with one dimensional center and we provide a full classification in this case for $G$ simple and classic.\par 
Shortly after this paper was posted, a new preprint of Della Vedova and Gatti (\cite{DVG}) appeared in the ArXiv, where the author prove interesting results 
concerning the construction of the almost complex structure (hence partially overlapping with section 2), the full classification of non compact Chern-Einstein manifolds when the group $G$ is simple of rank $\leq 4$ or exceptional as well as  some technique that might be useful in the project of a full classification for homogeneous non compact Chern-Einstein manifolds. These results can be seen as complementary to our main result Theorem \ref{MT}, which covers the generic case.  

\par \medskip

{\bf Notation.} Lie groups and their Lie algebras will be indicated by
capital and gothic letters respectively. We will denote the
Cartan-Killing form by $B$. If a Lie group $G$ acts on a manifold $M$, for every $X\in \gg$ we will denote by $X^*$ the corresponding vector field induced by the one-parameter subgroup $\exp(tX)$.
\section{Homogeneous almost K\"ahler non-compact manifolds of semisimple Lie groups}

Let $G$ be a non compact semisimple Lie group with Cartan decomposition
$$\gg = \gk + \gp,$$
where $\gk$ is the Lie algebra of a maximal compact subgroup $K$ and $\gp$ is an $\Ad(K)$-invariant complement with $[\gp,\gp]\subseteq \gk$.\par
We consider a homogeneous symplectic manifold of the form $G/L$ where $L\subseteq K$ is the centralizer $C_G(t_o)$ for some $t_o\in \gk$. Any $G$-homogeneous symplectic manifold with compact stabilizer is simply connected and it has this form (see e.g. \cite{BFR}). \par
The reductive decomposition of the manifold $M=G/L$ is given by
$$\gg = \gl + \gm = \gl + \gn + \gp, \quad \gk = \gl + \gn.$$
The manifold $G/L$ admits a $G$-equivariant fibration $G/L\to G/K$ over the non-compact symmetric space $S:= G/K$ with typical fibre given by the flag manifold $F:=K/L$. \par
Any invariant symplectic form $\o$ is defined by a closed non degenerate $Ad(L)$-invariant element $\o_\gm$ in $\Lambda^2(\gm^*)$. Any such form $\o_\gm$ can be written as $\o_\gm = d\eta$, where the 1-form $\eta= B\circ t_o$ for some element $t_o\in \gt$ so that $C_\gg(t_o) = \gl$. We fix an invariant symplectic form $\o$ which is associated to an element $t_o\in \gl$.\par
We denote by $\gt$ a Cartan subalgebra of $\gl$ and we set $\gh:= i\gt\subset \gg^c$ together with $z_o:=it_o\in\gh$.  The  complexification $\gg^{\small{\mathbb C}}$ has the root space decomposition
$$\gg^{\small{\mathbb C}} = \gh^{\small{\mathbb C}} \oplus\bigoplus_{\a\in R}\gg_\a,$$
where $R\subset \gh^*$ is the root system w.r.t. to the Cartan subalgebra $\gh^{\small{\mathbb C}}$.
The root system $R$ can be split as $R= R_{\gl}\cup R_\gm= R_\gl\cup ( R_c\cup R_{nc})$ where
$$\gl = \gh^{\small{\mathbb C}} \oplus\bigoplus_{\a\in R}\gg_\a,\
\gn^{\small{\mathbb C}}= \bigoplus_{\a\in R_c}\gg_\a,\ \gp^{\small{\mathbb C}}= \bigoplus_{\a\in R_{nc}}\gg_\a$$
The roots in $R_c$ are called {\it compact}, while roots in $R_{nc}$ are said to be {\it non-compact}. \par
We select a standard Chevalley basis for root spaces, namely a set of root vectors $\{E_\a\}_{\a\in R}$ so that $\gg_\a = \mathbb C\cdot E_\a$ for every $\a\in R$ and
$$B(E_\a,E_{-a}) = 1,\quad [E_\a,E_{-\a}] = H_\a\in \gh, $$
where $H_\a$ denotes the coroot in $\gh$ so that $B(H_\a,H)=\a(H)$ for every $H\in \gh$. If we now consider the antilinear involution $\sigma$ of $\gg^{\small{\mathbb C}}$ corresponding to the real form $\gg$ (for brevity we will write the conjugation $\tau$ as\ $\bar{}$\, ), then
$$\overline{E_{\a}}= - E_{-\a},\qquad \a\in R_c$$
$$\overline{E_{\a}}=  E_{-\a},\qquad \a\in R_{nc}.$$
We consider the vectors in $\gg$ defined as follows
$$v_\a := E_\a + \bar E_\a = \left\{\begin{matrix} E_\a - E_{-\a}, & \a\in R_c\\
E_\a + E_{-\a}, & \a\in R_{nc}\end{matrix}   \right.$$
$$w_\a := i(E_\a - \bar E_\a) = \left\{\begin{matrix} i(E_\a + E_{-\a}), & \a\in R_c\\
i(E_\a - E_{-\a}), & \a\in R_{nc}.\end{matrix}   \right.$$
Then
$$\gn = \bigoplus_{\a\in R_{c}} \mbox{Span}_{\mathbb R}\{v_\a,w_\a\},\qquad
\gp = \bigoplus_{\b\in R_{nc}} \mbox{Span}_{\mathbb R}\{v_\b,w_\b\}. $$
We now consider $G$-invariant almost complex structures $J$ on $G/L$, i.e. $\Ad(L)$-invariant endomorphisms $J\in \mbox{End}(\gm)$ with $J^2=-Id$, or equivalently $\ad(\gl^{\small{\mathbb C}})$-invariant endomorphisms $J\in \mbox{End}(\gm^{\small{\mathbb C}})$ with $J^2=-Id$ and commuting with the conjugation $\sigma$. \par
Since $\gl^{\small{\mathbb C}}$ contains a Cartan subalgebra, the $\ad(\gl)$-invariance of $J$ implies that $J$ preserves every root space and therefore for every $\a\in R_\gm$ we have
$$JE_\a = i \epsilon_\a \cdot E_\a,\qquad \epsilon_\a = \pm 1.$$
Therefore we can decompose
$$R_c = R_c^{10} \cup R_c^{01}, \qquad R_{nc} = R_{nc}^{10} \cup R_{nc}^{01},$$
where
$$R_{c/nc}^{10} =\{\a\in R_{c/nc}|\ \epsilon_{\a} = 1\},\ R_{c/nc}^{01} =\{\a\in R_{c/nc}|\ \epsilon_{\a} = - 1\}.$$
 Since $J$ commutes with $\sigma$, we have $\epsilon_{-\a} = - \epsilon_{\a}$ for every $\a\in R_\gm$, hence
$$R_c^{01} = - R_c^{10},\qquad R_{nc}^{01} = - R_{nc}^{10}.$$
The $\ad(\gl^{\small{\mathbb C}})$-invariance of $J$ means that
\beq\label{inv} (R_\gl + R_c^{10})\cap R \subseteq R_c^{10},\ (R_\gl + R_{nc}^{10})\cap R \subseteq R_{nc}^{10}.\eeq

We now consider the invariant pseudo-Riemannian metric $g$ which is defined by the symmetric form
$$g(u,v) = \o(u,Jv),\ u,v\in \gm.$$
\begin{lemma} The pseudo-Riemannian metric $g$ defined above is $J$-Hermitian and it is positive definite if and only if the following conditions are satisfied
	$$\a(z_o)>0\quad \mbox{{for}}\ \a\in R_c^{10},\qquad
	\a(z_o)<0\quad \mbox{{for}}\ \a\in R_{nc}^{10}. $$
	
\end{lemma}	
\begin{proof} In order to prove that $g$ is $J$-Hermitian, we note that $\o(E_\a,E_\b)\neq 0$ if and only if $\a+\b=0$ and in this case we have
	$$\o(JE_\a,JE_{-\a}) = - \epsilon_{-\a}\cdot\epsilon_{\a}\cdot\o(E_\a,E_{-\a}) =
	\o(E_\a,E_{-\a}).$$
Clearly $g(E_\a,E_\b)= 0$ if $\a+\b\neq 0$ and $g$ is positive definite if and only if $g(v_\a,v_\a) >0$ for every $\a\in R_\gm$.
Now if $\a\in R_c^{10}$ we have
$$0<g(v_\a,v_\a)= \o(E_\a-E_{-\a},J(E_\a-E_{-\a})) = i\ \o(E_\a-E_{-\a},E_\a+E_{-\a}) =
2i\ \o(E_\a,E_{-\a}) = $$
$$ = 2i\ B([t_o,E_\a],E_{-\a}) = 2\a(z_o),$$
while if $\a\in R_{nc}^{10}$ we have
$$0<g(v_\a,v_\a)= \o(E_\a+E_{-\a},J(E_\a+E_{-\a})) = i\ \o(E_\a+E_{-\a},E_\a-E_{-\a}) =
-2i\ \o(E_\a,E_{-\a}) = $$
$$ = -2i\ B([t_o,E_\a],E_{-\a}) = -2\a(z_o).$$
\end{proof}
We summarize the above arguments in the following Theorem.
\begin{theorem}\label{Thm1} Let $G$ be a semisimple non-compact Lie group, $L$ a compact subgroup of $G$ given by the centralizer in $G$ of some element $t_o\in \gg$. Let $\omega = \omega_{t_o}$ be the invariant symplectic form associated to $t_o$. Then there exists a unique extension of the homogeneous symplectic manifold $(M,\o)$ to a homogeneous almost K\"ahler manifold $(M,\o,J_{t_o})$ where the invariant almost complex structure $J_{t_o}$ is defined by the holomorphic space $\gm^{10}$
	$$\gm^{10} = \gg(R^{10}),\quad R^{10} = R_c^{10}\cup R_{nc}^{10},$$
	$$R_c^{10} = \{\a\in R_c|\ \a(z_o)>0\},\quad R_{nc}^{10} = \{\a\in R_{nc}|\ \a(z_o)<0\}.$$
The almost complex structure $J_{t_o}$ is integrable, hence $(M,\o,J_{t_o})$ is K\"ahler if and only if the symmetric space $G/K$ is Hermitian.
	\end{theorem}
The last assertion follows from the following observations. Indeed, $(R_c^{01}+R_{nc}^{10}) \cap R \subset R_{nc}^{10}$ and $K$-invariance of $J|_{\gp}$ is equivalent to the integrability condition $(R_c^{10}+ R_{nc}^{10})\cap R \subset R_{nc}^{10}$.

\begin{remark}\label{Decompodition of G} (1)\ Note that the restriction of the almost complex structure $J_{t_o}$ to the (complex) fibre $F$ is integrable and $(\o|_F,J|_F)$ is a K\"ahler structure. \par
	(2)\ If we decompose $G = G_1\cdot G_2\cdot \ldots\cdot G_k$ as the product of its simple factors, then the homogenous almost K\"ahler space $G/L$ splits accordingly as $\Pi_{i=1}^kG_i/L_i$ where $L_i = L\cap G_i$ and each factor is a homogeneous almost K\"ahler space. Therefore we can always assume that $G$ is simple.
\end{remark}
     Let    $M =  G/L  \to  G/K$   be  the $G$-equivariant   fibering    of  an  almost   K\"ahler homogeneous  manifold  $(M = G/L, \omega, J)$   over  the    symmetric  space $G/K$.	
 We can  construct an integrable complex structure   $\tilde J$ which coincides with $J$ along the fibre $F$ and is the opposite to $J$ on the orthogonal space.  We  call   $\tilde J$  the   complex  structure  associated  to the   almost   K\"ahler homogeneous  manifold  $(M = G/L, \omega, J)$.
  Note  that  $(\o,\tilde J_{t_o})$  is  an invariant pseudo-K\"ahler structure  on  $M = G/L$. \par

\medskip
\section{The Chern-Einstein equation for almost K\"ahler homogeneous manifolds}

Now we describe the expression for the Chern connection $D$ on the homogeneous almost K\"ahler manifold $G/L$ with invariant symplectic structure $\o_{z_o}$ and corresponding invariant almost complex structure $J$ (see also \cite{P}). \par
It is well known that the Chern connection is the unique connection $D$ that leaves $g$ and $J$ parallel and whose torsion $T$ satisfies the property
$$T(JX,Y) = T(X,JY) = JT(X,Y).$$
We are mainly interested in the (first) Ricci form $\rho$ which is a defined as
$$\rho(X,Y) = \mbox{{Tr}}\ J\circ R_{XY},$$
where $R$ denotes the curvature tensor. It is known that the Ricci form $\rho$ is a closed $2$-form whose cohomology class $[\rho]$ represents $2\pi c_1(M,\omega)$ (see e.g.~ \cite{AD},~\S 7).\par
We recall that any invariant connection $D$ on the homogeneous space $G/L$ with reductive decomposition $\gg = \gl + \gm$ can be described by the $\ad(\gl)$-equivariant Nomizu's map $\Lambda:\gg\to \mbox{{End}}(\gm)$ satisfying the condition
$$\Lambda_X = \ad_X|_\gm,\qquad X\in \gl.$$
Under the identification $\gm\cong T_oG/L$ we have
$$\Lambda_XY = (D_XY^* - [X^*,Y^*])|_o,$$
where $X^*,Y^*$ denote the vector fields on $M$ corresponding to $X,Y\in\gm$. Then in terms of Nomizu's operator, the torsion $T$ and the curvature $R$ at $o\in M$  are given by
$$T(X,Y) = \Lambda_XY-\Lambda_YX - [X,Y]_\gm,$$
$$R_{XY} = [\Lambda_X,\Lambda_Y] - \Lambda_{[X,Y]}.$$
If $D$ is the Chern connection on the homogeneous space  $G/L$ we can compute its Ricci form in terms of the root space decomposition.
\begin{proposition} For every root $\a,\b\in R_\gm$ we have
	$$\rho(E_\a,E_\b) = 0 \quad \mbox{{if}}\ \a+\b\neq 0,$$
	$$\rho(E_\a,E_{-\a}) = -2i\ \sum_{\beta\in R_\gm^{10}} \langle\a,\b\rangle.$$
\end{proposition}
\begin{proof} Using the expression for the curvature $R$ and the fact that $\Lambda$ commutes with $J$, we see that for $X,Y\in\gm$
$$\rho(X,Y) = Tr (J\Lambda_X\Lambda_Y - J\Lambda_Y\Lambda_X) - Tr(J\Lambda_{[X,Y]}) = $$	
$$= Tr ([\Lambda_X,J\Lambda_Y]) - Tr(J\Lambda_{[X,Y]}) =
- Tr(J\Lambda_{[X,Y]}).$$
For $\a,\b\in R_\gm$ and $H\in\gt$ the $\ad(\gl)$-invariance implies that
$$0 = \rho([H,E_\a],E_\b) + \rho(E_\a,[H,E_\b]) = (\a+\b)(H)\cdot \rho(E_\a,E_\b),$$
so that $\rho(E_\a,E_\b)=0$ unless $\a+\b=0$.\par
Now for $\a\in R_\gm$
$$\rho(E_\a,E_{-\a}) = -Tr|_{\gm^C}J\ad(H_\a) = -2i\sum_{\b\in
R_\gm^{10}} \langle \a,\b\rangle.$$
\end{proof}
We introduce the Koszul's form
$$\delta := 2\sum_{\gamma\in R_\gm^{10}} \gamma \in \gh^*$$
so that we have the following corollary
\begin{corollary} The Ricci form $\rho$ is given by
$$\rho = i\ d\delta,$$
where for every $X,Y\in \gm$
$$d\d(X,Y) = -\d([X,Y]_{\gt}).$$
The  Ricci  form  $\rho$ does  not  depend  on  the metric,  but  only on  the   almost   complex structure  $J$

\end{corollary}

Note that the map $d:\gt^*\to \Lambda^2(\gm^*)$ is injective and $\rho$ is $\Ad(L)$-invariant, so that  $\d$ belongs to the center $\gz$ of $\gl$. \par
\begin{definition} An almost K\"ahler manifold $(M,\o,J)$ is called Chern-Einstein if its Ricci form $\rho$ satisfies
	$$\rho = \lambda \omega$$
	for some constant $\l\in \mathbb R$.
\end{definition}
\begin{remark} We remark that when $L$ is maximal compact in $G$, i.e. the homogeneous space $G/L$ is Hermitian symmetric, the center of $L$ is one-dimensional and there exists only one invariant symplectic structure up to a multiple. It follows that the corresponding invariant almost complex structure is integrable and the manifold is K\"ahler. In this case the Chern connection coincides with the Levi Civita connection and the manifold is K\"ahler-Einstein.

More generally, it is known that any  non-compact homogeneous K\"ahler manifold $G/L$ with  non compact simple $G$  is K\"ahler  if  a   maximal  compact  subgroup   $K \supset L$ of  $G$  has  1-dimensional  center.   Then
 $G/L \to  G/K$ is  a  $G$-equivariant  fibering  over   the  Hermitian  symmetric  space  $G/K$.  Moreover,  $G/K$  is  the only   K\"ahler-Einstein homogeneous  manifold   of the  group  $G$,
         see e.g. \cite{BFR}. \end{remark}

    Let $\tilde J$  be  the  integrable  complex   structure,    associated    with    $(M=G/L, \o,  J)$.
 We  set 
 $$ \tilde{\delta}= 2\sum_{\alpha \in R_{\mathfrak{m}} ,  \alpha(z_0) >0}\alpha.$$
     Then
    $\tilde{\rho} =  i  d \tilde{\delta} $  defines   an invariant non-degenerate  representative $\tilde{\rho}$   of   the Chern  class $c_1(\tilde J)$.  Hence   for    any  $\lambda \neq 0$,   $\tilde{\o}_{\l}  := -(\l)^{-1}\tilde{\rho}$  defines an invariant   pseudo-K\"ahler  structure   $(\tilde{\o}_{\l}, \tilde{J})$
  on  $M =  G/L$   which  satisfies  the Einstein  equation    $\tilde{\rho} = \l \tilde{\o}_{\l} $.
   
\begin{proposition}	
	Let  $(M=G/L,  \o, J)$ be  a  homogeneous   almost K\"ahler  manifold  of  a  semisimple  Lie  group  $G$. Then   $(\tilde{\o}_{\l}, \tilde J)$    is   an invariant pseudo-K\"ahler-Einstein    structure  on  $M = G/L$. \end{proposition}	
	
\section{ Chern-Einstein almost K\"ahler homogeneous spaces }

\subsection{General approach for classification}
Let $(M=G/L,\o,J)$ be a homogeneous almost K\"ahler manifold with $G$ simple non-compact Lie group.  The description of all almost K\"ahler Chern Einstein homogeneous manifolds with $\rho = \l \o $ ($\l\in \mathbb R$) reduces to the solutions of the following equation
$$\d = \l (B\circ z_o).$$
Recall that a real simple Lie algebra is either the real form of a complex simple Lie algebra or it is the realification of a complex simple Lie algebra. The following Lemma shows that the last possibility does not occur.
\begin{lemma} If $M=G/L$ with $G$ simple admits an invariant almost K\"ahler structure, then $\gg$ is the real form of a complex simple Lie algebra. \end{lemma}
\begin{proof}  If $\gg = \gs_\mathbb R$, where $\gs$ is a simple complex Lie algebra, then a Cartan decomposition of $\gg$ is given by $\gg = \gq + i\gq$, where $\gq$ is a compact real form of $\gs$. If $z\in\gg$ is an element whose centralizer $\gl$ is a compact subalgebra, then there is an
automorphism $\phi$ of $\gg$ such that $\phi(z)\in\phi(\gl)\subseteq  \gq$. But then the centralizer of $\phi(z)$ in $\gg$ is given by $\phi(\gl)+ i\phi(\gl)$, a contradiction.\end{proof}
Therefore we will suppose that $\gg$ is a real form of the complex simple Lie algebra $\gg^c$. \par
{\bf Step 1.}\ As a first step we consider the list of all inner symmetric pairs $(\gg,\gk)$ of non-compact type with $\gg$ simple. Using the notation as in \cite{He}, p.~126, we obtain Table 1.  \par
\begin{table}[ht]\label{T1}
	\centering
	\renewcommand\arraystretch{1.1}
	\begin{tabular}{|c|c|c|c|}
		\hline
		{\mbox{Type}}				& 	$\gg$					&	$\gk$	& 	{\mbox{conditions}}			 	\\ \hline \hline
	$A$ &	$\su(p,q)$ & $\su(p) + \su(q) +\mathbb R$ & $p\geq q\geq 1$				\\ \hline
	$B$ & $\so(2p+1,2q)$ & $\so(2p+1) + \so(2q)$ & $p\geq 0,q\geq 1$			\\ \hline
	$C$ & $\gsp(n,\mathbb R)$ & $\su(n)+\mathbb R$ & $n\geq 1$			\\ \hline
	$C$ & $\gsp(p,q)$ & $\gsp(p) + \gsp(q)$ & $p,q\geq 1$			\\ \hline
	$D$ & $\so(2n)^*$ & $\su(n)+\mathbb R$ & $n\geq 3$			\\ \hline
	$D$ & $\so(2p,2q)$ & $\so(2p) + \so(2q)$ & $p,q\geq 1,p+q\geq 3$			\\ \hline
	$G$ & $\gg_{2(2)}$ & $\su(2)+\su(2)$ & 			\\ \hline
	$F$ & $\gf_{4(-20)}$ & $\so(9)$ & 			\\ \hline
	$F$ & $\gf_{4(4)}$ & $\su(2)+\gsp(3)$ & 			\\ \hline
	$E$ & $\gge_{6(2)}$ & $\su(2)+\su(6)$ & 			\\ \hline
	$E$ & $\gge_{6(-14)}$ & $\so(10)+\mathbb R$ & 			\\ \hline
	$E$ & $\gge_{7(7)}$ & $\su(8)$ & 			\\ \hline
	$E$ & $\gge_{7(-5)}$ & $\su(2)+\so(12)$ & 			\\ \hline
	$E$ & $\gge_{7(-25)}$ & $\gge_6+\mathbb R$ & 			\\ \hline
	$E$ & $\gge_{8(8)}$ & $\so(16)$ & 			\\ \hline
	$E$ & $\gge_{8(-24)}$ & $\su(2)+\gge_7$ & 			\\ \hline
	\end{tabular}
	\vspace{0.1cm}
	\caption{Inner symmetric pairs $(\gg,\gk)$ of non-compact type with $\gg$ simple.}\label{ncptSHF}
\end{table}
\renewcommand\arraystretch{1}
\par
{\bf Step 2.}\ We fix a Cartan subalgebra $\gt$ in $\gk$ and we choose an {\it admissible} element $z\in \gh:=i\gt$, i.e. such that
$$C_{\gk}(t_o) = C_{\gg}(t_o) := \gl.$$
{\bf Step 3} We define
$$R_c^+(z) := R_c^{10}(z) = \{\a\in R_c|\ \a(z)>0\},$$
$$R_{nc}^+(z) := R_{nc}^{01}(z) = \{\a\in R_{nc}|\ \a(z) >0\}$$
and we set
$$\d_c(z) := 2\sum_{\a\in R_c^+(z)}\a,\qquad
\d_{nc}(z) := 2\sum_{\a\in R_{nc}^+(z)}\a,$$
$$\d(z) := \d_c(z) - \d_{nc}(z).$$
{\bf Step 4.} We solve the equation
\beq \label{equat}\d(z) = \l Bz.\eeq
\subsection{Examples of Chern Einstein manifolds}\label{Ex} In the special case when $\gl$ has one-dimensional center $\gz = i\mathbb R z$, then the equation \ref{equat} is automatically satisfied for some $\l\in\mathbb R$ since $\d(z)$ belongs to the center of $\gl$.
In the particular case when $\gl=\gk$, the space is K\"ahler and irreducible Hermitian symmetric, hence automatically Chern-Einstein. We may then  start with a non-Hermitian symmetric pair $(\gg,\gk)$ out of Table 1 and all simple factors of $\gk$ will be included in $\gl$ except one, say $\gk_1$. In $\gk_1$ we take $\gl_1$ so that the pair $(\gk_1,\gl_1)$ is a flag manifold whose corresponding painted Dynkin diagram has only one black node (see e.g.~\cite{AP},~\cite{BFR}).
We then have to restrict ourselves to the cases when the centralizer in $\gg$ of the center of $\gl_1$ coincides with $\gl$. It is not difficult to see that for classical $\gg$ the only pairs $(\gg,\gl)$ with $\dim\gz(\gl) = 1$  are given in Table 2.
\begin{table}[ht]\label{T2}
	\centering
	\renewcommand\arraystretch{1.1}
	\begin{tabular}{|c|c|c|c|}
		\hline
		{\mbox{Type}}				& 	$\gg$					&	$\gl$	& 	{\mbox{conditions}}			 	\\ \hline \hline
	
		$B$ & $\so(2p+1,2q)$ & $\mathbb R + \so(2p+1) + \su(q)$ & $p\geq 0,q\geq 1$			\\ \hline
		$C$ & $\gsp(p,q)$ & $\begin{matrix} \bullet\ \mathbb R + \su(p) + \gsp(q)\\\bullet\ \mathbb R + \gsp(p) + \su(q)\end{matrix}$  & $p,q\geq 1$			\\ \hline
		$D$ & $\so(2p,2q)$ & $\begin{matrix} \bullet\ \mathbb R + \su(p) + \so(2q)\\\bullet\ \mathbb R + \so(2p) + \su(q)\end{matrix}$ & $p,q\geq 1,p+q\geq 3$			\\ \hline
		
	\end{tabular}
	\vspace{0.1cm}
	\caption{Pairs $(\gg,\gl)$  with $\gg$ simple of classical type, $\gl\subsetneq\gk$ and $\dim\gz(\gl)=1$.}\label{ncptSHF}
\end{table}
\renewcommand\arraystretch{1}
\par

Note also that in \cite{DV},~ Theorem 5, the homogeneous manifolds $M = \SO(2p,q)/\U(p)\times \SO(q)$ are shown to be Chern-Einstein with constant $\l = 2(p-q-1)$, which therefore may vanish or assume a positive/negative sign.

\subsection{The case of abelian $L$}
We now restrict ourselves to the case when the isotropy $L$ is a maximal torus. We recall that for any admissible $z\in \gh$ the set
$$R_c^+(z) \cup R_{nc}^+(z)$$
gives a sytem of positive roots of $\gg^C$ and therefore there exists a unique element $w$ in the Weyl group $W$ that maps the standard system of positive root $R_o^+$ into  $R_c^+(z) \cup R_{nc}^+(z)$. \par
Then
$$\d_c(z) = 2\sum_{\a\in w(R_o^+)\cap R_c} \a, \quad \d_{nc}(z)= 2\sum_{\a\in w(R_o^+)\cap R_{nc}} \a.$$
Note that the equation \eqref{equat} with $\l\neq 0$ admits a solution if and only
\beq\label{cond} \forall \a\in w(R_o^+)\quad  \langle \a,\d(z)\rangle > 0\ (\mbox{{or}}\ <0) \qquad \mbox{{if}}\ \l>0\ (\l<0\ \mbox{{resp.}}).\eeq
We remark that the computation of $\d(z)$ depends only on the root system of $\gg^C$ and its decomposition into compact and non-compact roots together with the action of Weyl group.

In the next sections we will go through the classical Lie algebras of type A,B,C,D.

\subsection{Proof of the main Theorem in case $\gg$ of type $A_n$}

According to Table 1, we analyze the case $\gg = \su(p,q)$, $p\geq q\geq 1$. The standard Cartan subalgebra $\gt\subset \gk$ gives rise to the root system $R$ of $\gg^c = \gsl(n+1,\mathbb C)$, where $p+q=n+1$, given by $\{\epsilon_i-\epsilon_j|\  i,j=1,\ldots,n+1,\ i\neq j\}$. The standard system of positive roots $R_o^+$ is given by $\{\epsilon_i-\epsilon_j|\ i<j\}$ and the Weyl group $W$ is given by the full group of permutations $\mathcal S_{n+1}$. \par
If we put $P:=\{1,\ldots, p\}$ and $Q:=\{p+1,\ldots,n+1\}$, we have
$$R_c=\{\e_i-\e_j|\ i,j\in P\ \mbox{or}\ i,j\in Q, i\neq j \},\quad R_{nc} = R\setminus R_c.$$
We consider an element $\s$ of the Weyl group, i.e. $\sigma\in \mathcal S_{n+1}$, and we define $P_\s = \s^{-1}(P)$, $Q_\s := \s^{-1}(Q)$. Then
$$\s(R_o^+)\cap R_{c} = \{\es i - \es j|\ i,j\in P_\s,\ i<j\}\cup
\{\es i - \es j|\ i,j\in Q_\s,\ i<j\}, $$
$$\s(R_o^+)\cap R_{nc} = \{\es i - \es j|\ i<j,\ i\in P_\s, j\in Q_\s\}\cup
\{\es i - \es j|\ i<j,\ i\in Q_\s, j\in P_\s\}. $$
For $i\in \{1,\ldots,n+1\}$ we set

$$ k_Q(i) := |\{k\in Q_\s|\ k >i\}|,\quad k_P(i) := |\{k\in P_\s|\ k >i\}|,$$
$$ \bar k_Q(i) := |\{k\in Q_\s|\ k <i\}|,\quad \bar k_P(i) := |\{k\in P_\s|\ k <i\}|.$$
We have

$$\sum_{\a\in \s(R^+)\cap R_{c}} \a = \sum_{\begin{smallmatrix} i<j& \\ i,j\in P_\s& \end{smallmatrix}} \es i - \es j\ + \sum_{\begin{smallmatrix} i<j& \\ i,j\in Q_\s& \end{smallmatrix}} \es i - \es j=$$
$$ = \sum_{i\in P_\s}(k_P(i)- \bar k_P(i))\ \es i +
\sum_{i\in Q_\s}(k_Q(i) - \bar k_Q(i))\ \es i$$
and similarly we obtain
$$\sum_{\a\in \s(R^+)\cap R_{nc}} \a = \sum_{i\in P_\s}(k_Q(i)-\bar  k_Q(i))\ \es i +
\sum_{i\in Q_\s}(k_P(i) - \bar k_P(i))\ \es i.$$
Therefore for any admissible $z\in\gh$ with corresponding system of positive roots given by $\sigma(R_o^+)$ we obtain
\beq \label{d}\frac 12 \d(z) = \sum_{i\in P_\s}(k_P(i)- \bar k_P(i) - k_Q(i)+\bar  k_Q(i))\ \es i +\eeq
$${}\qquad +\sum_{i\in Q_\s}(k_Q(i) - \bar k_Q(i)-k_P(i) + \bar k_P(i))\ \es i.$$

\begin{lemma}\label{L1} (i)\ If $i\in P_\s$ we have
	$$k_P(i)- \bar k_P(i) - k_Q(i)+\bar  k_Q(i) = 4k_P(i) + 2i - n - 2p.$$
	(ii)\  If $i\in Q_\s$ we have
	$$k_Q(i) - \bar k_Q(i)-k_P(i) + \bar k_P(i) = -4k_P(i)-2i+n+2p +2.$$
\end{lemma}
\begin{proof} We start noting the following trivial equalities for all $i=1,\ldots,n+1$
	\beq\label{rel}k_P(i) + k_Q(i) = n+1-i,\quad \bar k_P(i) + \bar k_Q(i) = i-1.\eeq
	
	If $i\in P_\s$ then $k_P(i)+\bar k_P(i) = p-1$ and therefore the first claim follows using \eqref{rel}. The second claim follows similarly using that $ k_P(i)+\bar k_P(i) = p$ for $i\in Q_\s$.\end{proof}
We now solve the equation \eqref{equat} for $\l\neq 0$. \par \medskip
{\bf Case  $\l>0$.}\ Condition \eqref{cond} together with \eqref{d} and Lemma \ref{L1} imply the following: if $i,j\in P_\s$, $i<j$,
\beq\label{P} 2k_P(i) + i > 2k_P(j) + j,\eeq
while if $i,j\in Q_\s$, $i<j$,  we have
\beq \label{Q} 2k_P(j)+ j > 2k_P(i) + i.\eeq
Now suppose $i,j\in P_\s$ with $i<j$ and any $i<a<j$ belongs to $Q_\s$. Then
$k_P(i) = k_P(j) + 1$ and \eqref{P} implies $2> j-i$, i.e. $j=i+1$. This means that $P_\s$ is made of consecutive numbers. If we repeat the same argument with $Q_\s$ we obtain that also $Q_\s$ is made of consecutive numbers. Therefore we are left with the following two possibilities:
\begin{enumerate}
\item[(a)] $P_\s =\{1,\ldots,p\},\ Q_\s=\{p+1,\ldots,n+1\},$
\item[(b)] $P_\s=\{q+1,\ldots,n+1\},\ Q_\s = \{1,\ldots,q\}.$
\end{enumerate}
We now consider condition \eqref{cond} with $\a\in \s(R_o^+)\cap R_{nc}$. In case (a) we choose $\a = \e_{\s(p)}-\e_{\s(p+1)}$ and \eqref{cond} gives $-2n>0$, a contradiction. In case (b) we choose
$\a = \e_{\s(q)}-\e_{\s(q+1)}$ and \eqref{cond} gives again $-2n>0$, a contradiction. \par \medskip
{\bf Case $\l<0$.}\ For every $i,j\in P_\s$, $i<j$, we have
$$2k_P(i)+i < 2k_P(j) + j.$$
We claim that $p=1$. Indeed, suppose there exist $i<j$ in $P_\s$ with
$(i,j)\cap P_\s=\emptyset$. Then $k_P(i)= k_P(j) +1$ and therefore $j-i>2$. This means that there exist at least two consecutive numbers, say $l,l+1$ in $Q_\s$. Then $k_P(l) = k_P(l+1)$ and condition \eqref{cond} with $\a = \e_{\s(l)}-\e_{\s(l+1)}$ gives the contradiction $l>l+1$. This implies that $P_\s$ has only one element, i.e. $p=1$. Hence $1=q\leq p=1$. In this case $R_c=\emptyset$ and $\gg = \gsl(2,\mathbb R)$, $\gl = \mathbb R$. Since $R = R_{nc}$ consists of one root up to sign, the equation $\rho=\l\o$ is satisfied for some $\l<0$.\par \medskip
{\bf Case $\l=0$.}\ We recall that $\sum_{i=1}^{n+1}\e_i=0$ is the only relation among the $\{\e_i\}_{i=1,\ldots,n+1}$, so that the equation $\d(z)=0$ implies that all the coefficients in the right-hand side of \eqref{d} are mutually equal.
\begin{lemma} Two elements in $P_\s$ or $Q_\s$ are not consecutive.\end{lemma}
\begin{proof} Indeed, suppose $i,i+1\in P_\s$. Then $k_P(i)=k_P(i+1)+1$ and
	$$4k_P(i+1)+2(i+1) = 4k_P(i) + 2i,$$
	yielding a contradiction. A similar argument applies for $Q_\s$.\end{proof}
As a corollary, we see that $n$ is even. Indeed, consider $i\in P_\s$ and $i+1\in Q_s$. Then $k_P(i) = k_P(i+1) +1$ and using Lemma \ref{L1} we have
$$8k(i) + 4i -2n -4p -4 = 0,$$
that implies $n$ is even. Now $P_\s$ and $Q_\s$ consist precisely of the sets of even or odd integers in $\{1,\ldots,n+1\}$. Since $n$ is even, the set of even (odd) numbers less or equal to $n+1$ has cardinality $\frac n2$ ($\frac n2 + 1$ resp.) and since we supposed $p\geq q$, we have $p=q+1$ and
$$P_\s = \{1,3,5,\dots,n+1\},\quad Q_\s = \{2,4,6,\ldots,n\}.$$
Viceversa it is immediate to check that with the above choice of $P_\s,Q_\s$ the equation $\d(z)=0$ is satisfied. Notice that a suitable permutation $\s$ is given by $\s(2k) = p+k,\s(2k+1)=k+1$ for $k=1,\ldots,q$.
\begin{example} We consider $\gg = \su(3,2)$. A permutation $\s$ that produces a Chern-Ricci flat almost K\"ahler structure on $\SU(3,2)/{\mbox{T}}^4$ is given by the cycle $(2453)$ and the corresponding system of positive roots is given as follows (here $\e_{ij} := \e_i-\e_j$ for the sake of brevity):
$$R_c^+(z) = \{\e_{12},\e_{13},\e_{23},\e_{45}\},\
R_{nc}^+(z) = \{\e_{14},\e_{15},\e_{25},\e_{42},\e_{43},\e_{53}\}.$$ \end{example}

\subsection{Proof of the main Theorem in case $\gg$ of type $B_n$} According to Table 1, when $\gg^c = \so(2n+1,\mathbb C)$, $n\geq 2$, we consider the subalgebras $\gg = \so(2p+1,2q)$ with $p+q=n$, $p,q\geq 1$ and $\gg = \so(1,2n)$ separately. The standard root system $R$ is given by
$R = \{\pm \e_i,\ \pm \e_i\pm \e_j,\ 1\leq i\neq j\leq n\}$ ($\pm$ independent), with $R_o^+ = \{\e_i,\ \e_i\pm\e_j,\ 1\leq i<j\leq n\}$. \par
We start considering the case $\gg = \so(2p+1,2q)$ with $p+q=n$, $p,q\geq 1$.
The compact roots are given by
$$R_c = \{\pm \e_i,\ i=1,\ldots,p\} \cup \{\pm\e_i\pm\e_j,\ 1\leq i\neq j\leq p\} \cup \{\pm\e_i\pm\e_j,\ p+1\leq i\neq j\leq n\}.$$
An element $w$ in the Weyl group $W\cong \mathbb (Z_2)^n\rtimes \mathcal S_n$ acts as $w(\e_i) = \phi_i\e_{\s(i)}$, where $\phi_i\in \{1,-1\}$ can be chosen independently. Therefore an easy computation shows that
\beq
\frac 12 \d(z) = \sum_{i\in P_\s}(4k_P(i)+2i-2n+1) \phi_i\e_{\s(i)} +
\sum_{i\in Q_\s} (4k_Q(i)+2i - 2n -1) \phi_i\e_{\s(i)},\eeq
where $P_\s := \s^{-1}\{1,\ldots,p\}$, $Q_\s := \s^{-1}\{p+1,\ldots,n\}$ and
$k_{P/Q}(i):= |\{j\in P_\s/Q_\s|\ j>i \}|$.\par \medskip\textit{}
{\bf Case $\l>0$.}\  If $i\in Q_\s$, then $\a:= \phi_i\e_{\s(i)}\in w(R_o^+)$ and
$\langle \d(z),\a\rangle >0$ implies $4k_Q(i)+2i > 2n +1$. If $i_Q$ is the maximum element in $Q_\s$, then $k_Q(i_Q)=0$ and therefore $2i_Q>2n+1$, a contradiction.\par \medskip
{\bf Case $\l<0$.}\ The maps $c_P:P_\s\ni i\mapsto 4k_P(i)+2i-2n+1$ and
$c_Q:Q_\s\ni j\mapsto 4k_Q(i)+2i - 2n -1$ are negative and strictly increasing. It follows that if $i< j$ are two numbers both in $P_\s$ or both in  $Q_\s$ then $j-i>2$, contradicting $P_\s\cup Q_\s=\{1,\ldots,n\}$. This implies that $p=q=1$. Since $2=n\not\in P_\s$ (otherwise $c_P(2) = 1$), we see that $P_\s=\{1\}$ and $Q_\s=\{2\}$. In this case $c_P(1)=-1=c_Q(2)$ and therefore we contradict \eqref{cond} using $\a= \phi_1\e_1-\phi_2\e_2$.\par \medskip{\bf Case $\l=0$.}\ This cannot occur as all the coefficients of $\frac 12\d(z)$ are odd numbers.\par \medskip
We now deal with the case $\gg = \so(1,n)$, where the compact roots are given by $R_c = \{\pm \e_i \pm \e_j,\ 1\leq i\neq j\leq n\}$. It is immediate to compute
$$\frac 12 \d(z) = \sum_{i=1}^n (2n-2i-1)\phi_i\es i.$$
Since the coefficients $2n-2i-1$ do not have a constant sign for $i=1,\ldots,n$ and cannot vanish, we see that equation \eqref{equat} has no solution.
\par \medskip

\subsection{Proof of the main Theorem in case $\gg$ of type $C_n$}

When $\gg$ is a real form of $\gsp(n,\bC)$, $n\geq 3$, we need to consider two subcases according to Table 1, namely when $\gk = \gu(n)$ or $\gk = \gsp(p) + \gsp(q)$, $n=p+q$. \par
The standard root system of $\gsp(n,\bC)$ is given by $R = \{\pm \e_i \pm \e_j,\  1\leq i< j\leq n\}\cup\{\pm 2\e_i,\ i= 1,\ldots,n\}$ with
$R_o^+ = \{ \e_i\pm \e_j, 2\e_i,\ i<j,\ i,j=1,\ldots,n  \}$. The Weyl group of $\gg^c$ is a semidirect product $(\bZ_2)^n\rtimes \mathcal S_n$ and any $w\in W$ acts as $w(\e_i) = \phi_i\e_{\s(i)}$, where $\phi_i\in \{1,-1\}$ can be chosen independently.
\subsubsection{Case $\gk = \gu(n)$.}


We have $$R_c = \{ \e_i - \e_j,\ i\neq j = 1,\ldots, n\}.$$
 Given $w\in W$ we have
$$w(R_o^+)\cap R_c = \{\phi_i\es i + \phi_j\es j|\ i<j,\ \phi_i\phi_j < 0\} \cup
\{\phi_i\es i - \phi_j\es j|\ i<j,\ \phi_i\phi_j > 0\}$$
and therefore, if we denote $A:= \{i|\ \phi_i=1\}$, $B:= \{i|\ \phi_i=-1\}$,
$$\sum_{\a\in w(R_o^+)\cap R_c}\a = \sum_{\begin{smallmatrix} i<j& \\ \phi_i\phi_j<0& \end{smallmatrix}}
\phi_i\es i + \sum_{\begin{smallmatrix} i<j& \\ \phi_i\phi_j<0& \end{smallmatrix}}
\phi_j\es j +\sum_{\begin{smallmatrix} i<j& \\ \phi_i\phi_j>0& \end{smallmatrix}}
\phi_i\es i - \sum_{\begin{smallmatrix} i<j& \\ \phi_i\phi_j>0& \end{smallmatrix}}
\phi_j\es j = $$

$$= \sum_{i\in A}\sum_{B\ni j>i} \es i - \sum_{i\in B}\sum_{A\ni j>i} \es i +
\sum_{j\in A}\sum_{B\ni i<j} \es j - \sum_{j\in B}\sum_{A\ni i<j} \es j  + $$
$$+ \sum_{i\in A}\sum_{A\ni j>i} \es i -  \sum_{i\in B}\sum_{B\ni j>i} \es i -
\sum_{j\in A}\sum_{A\ni i<j} \es j + \sum_{j\in B}\sum_{B\ni i<j} \es j.$$
We set
$$k_B(i):= |\{k\in B| k> i\}|,\ k_A(i):= |\{k\in A| k> i\}|,$$
$$\bar k_B(i) := |\{k\in B| k< i\}|,\ \bar k_A(i):= |\{k\in A| k< i\}|.$$
Hence
$$\sum_{\a\in w(R_o^+)\cap R_c}\a = \sum_{i\in A}k_B(i)\ \es i - \sum_{i\in B}k_A(i)\ \es i
+ \sum_{j\in A}\bar k_B(j) \es j - \sum_{j\in B}\bar k_A(j) \es j + $$
$$ + \sum_{i\in A}k_A(i)\ \es i -  \sum_{i\in B}k_B(i)\ \es i - \sum_{j\in A}\bar k_A(j)\ \es j + \sum_{j\in B}\bar k_B(j) \es j =$$
$$= \sum_{i\in A}(k_B(i) + \bar k_B(i)+ k_A(i) - \bar k_A(i))\ \es i + \sum_{i\in B}
(\bar k_B(i) - k_A(i) - \bar k_A(i) - k_B(i))\ \es i = $$
$$=\sum_{i\in A}(-2\bar k_A(i)+n-1)\ \es i + \sum_{i\in B}
(2\bar k_B(i) - n+1)\ \es i.$$
Using the fact that $\sum_{\a\in R_o^+}w\a = 2\sum_{i=1}^n (n-i+1)\ \phi_i\es i$, we obtain
 $$\frac12 \d(z) = \sum_{i\in A}[-4\bar k_A(i)+2i-4]\ \es i + \sum_{i\in B}
 [4\bar k_B(i) -2i +4]\ \es i.$$
 Denote by $c_A:= -4\bar k_A(i)+2i-4$ and $c_B:= 4\bar k_B(i) -2i +4$ the two coefficients. We now discuss the equation \eqref{equat}.\par
 First suppose $\l>0$: using the root $\phi_i\es i\in w(R_o^+)$ and \eqref{cond}, we see that $c_A(i)>0$ if $i\in A$ and $c_B(i)<0$ if $i\in B$. Now, if $1\in A$, then $\bar k_A(1)=0$ and $c_A(1)<0$, while if $1\in B$ then $c_B(1)>0$, showing that $1\not\in A\cup B$, a contradiction.\par
 If $\l<0$, then similarly as above we have $c_A(i)<0$ if $i\in A$ and $c_B(i)>0$ if $i\in B$. Suppose $A$ is not empty and let $i_A$ be the minimum element in $A$: then $i_A<2$, i.e. $i_A=1$. Similarly, if $B$ is not empty, its minimum point is $1$, showing $1\in A\cap B$, a contradiction. Then either $A$ or $B$ is empty. If $B=\emptyset$, then $\bar k_A(i)=i-1$ and $c_A(i)=-2i$. Using the roots $\a=\es i - \es {i+1}$ and \eqref{cond}, we see that $c_A$ is increasing, a contradiction. Similarly $A=\emptyset$
  leads to a contradiction.\par
 Finally, $\l=0$ is also impossible, as $c_A(i)=0,\ i\in A$ and $c_B(i) = 0,\ i\in B$ force $i\in A\cup B$ to be even, a contradiction.
 \par

\subsubsection{Case $\gk=\gsp(p)+\gsp(q)$, $p\geq q\geq 1$.} If we denote $P := \{1,\ldots,p\}$ and $Q:= \{p+1,\ldots,n\}$ we have
$$R_c = \{ \pm \e_i \pm \e_j,\ \pm 2\e_i\ |\ i,j\in P\} \cup \{ \pm \e_i \pm \e_j,\ \pm 2\e_i\ |\ i,j\in Q\}$$
and therefore, if $P_\s := \s^{-1}(P), Q_\s := \s^{-1}(Q)$,
$$w(R_o^+)\cap R_c = \{\phi_i\es i \pm \phi_j\es j,\  2\phi_i\es i|\ i,j\in P_\s, i<j\}\ \cup  $$
$$\cup\ \{\phi_i\es i \pm \phi_j\es j,\  2\phi_i\es i|\ i,j\in Q_\s, i<j\}.
$$
Therefore
$$\sum_{\a\in w(R_o^+)\cap R_c}\a = 2 \sum_{i\in P_\s} (k_P(i)+1)\phi_i\es i + 2 \sum_{i\in Q_\s} (k_Q(i)+1)\phi_i\es i,$$
where $k_{P/Q}(i) := |\{k\in P_\s/Q_\s|\ k>i \}|$. We then obtain
$$\frac 12 \d(z) = \sum_{i\in P_\s} (4k_P(i)+2i-2n+2)\phi_i\es i +  \sum_{i\in Q_\s} (4k_Q(i)+2i-2n+2)\phi_i\es i.$$
Denote by $c_P:= 4k_P(i)+2i-2n+2$ and $c_Q:= 4k_Q(i)+2i-2n+2$ the two coefficients for $i\in P_\s, Q_\s$ respectively. We now discuss the equation \eqref{equat}.\par
First suppose $\l>0$: we have $c_P,c_Q>0$ and if we denote by $i_P$ the maximum element in $P_\s$ we have $i_P>n-1$, i.e. $i_P=n$. Similarly $n\in Q_\s$, contradicting the fact that $P_\s$ and $Q_\s$ are disjoint.\par
Suppose now $\l<0$. Then $c_P,c_Q$ are negative and strictly increasing. If $i< j$ are two numbers both in $P_\s$ or both in  $Q_\s$ then $j-i>2$, contradicting $P_\s\cup Q_\s=\{1,\ldots,n\}$. This implies that $p=q=1$. Then $c_P(i)=2i-2<0$ for $i\in P_\s$ and $c_Q(i)=2i-2<0$ for $i\in Q_\s$, a contradiction.\par
If now $\l=0$, we see that $c_P=0$ and $c_Q=0$ imply that $i\equiv n-1\ (\mbox{{mod}}\ 2)$ for every $i=1,\ldots,n$, a contradiction.

\subsection{Proof of the main Theorem in case $\gg$ of type $D_n$, $n\geq 3$}

When $\gg$ is a real form of $\so(2n,\bC)$, $n\geq 3$, we need to consider two subcases according to Table 1, namely when $\gk = \gu(n)$ or $\gk = \so(2p) + \so(2q)$, $n=p+q$. \par
The standard root system of $\so(2n,\bC)$ is given by $R = \{\pm \e_i \pm \e_j,\  1\leq i< j\leq n\}$ with
$R_o^+ = \{ \e_i\pm \e_j,\ i<j,\ i,j=1,\ldots,n  \}$. The Weyl group of $\gg^c$ is a semidirect product $(\bZ_2)^n\rtimes \mathcal S_n$ and any $w\in W$ acts as $w(\e_i) = \phi_i\e_{\s(i)}$, where $\phi_i\in \{1,-1\}$ satisfies $\prod_{i=1}^n\phi_i = 1$. \par
\subsubsection{Case $\gk=\gu(n)$.} This case can be dealt with similar arguments as in the subsection 4.6.1 for $\gg^c$ of type $C_n$ and $\gk = \gu(n)$ and we will omit the detailed computations, keeping the same notation. We have
$$\frac 12 \d(z) = \sum_{i\in A}[-4\bar k_A(i)+2i-2]\ \es i + \sum_{i\in B}
[4\bar k_B(i) -2i +2]\ \es i.$$
Denote by $c_A:= -4\bar k_A(i)+2i-2$ and $c_B:= 4\bar k_B(i) -2i +2$ the two coefficients. \par
If $\l>0$, then $c_A(i)>0$ if $i\in A$ and $c_B(i)<0$ if $i\in B$. It follows that $1\not\in A\cup B$, a contradiction.\par
If $\l<0$, then $c_A(i)<0$ if $i\in A$. If $i_A$ is the minimum point of $A$, then $\bar k_A(i_p)=0$ and therefore $i_p<1$, a contradiction. Then $A$ is empty. Similarly the minimum point $i_B$ of $B$ satisfies $i_B<1$, showing that $B$ is empty, a contradiction.\par
Finally, $\l=0$ is also impossible, as $c_A(i)=0,\ i\in A$ and $c_B(i) = 0,\ i\in B$ force every $i\in A\cup B$ to be odd, a contradiction.

\subsubsection{Case $\gk=\so(2p)+\so(2q)$, $p\geq q\geq 1$, $p+q\geq 3$.} If we denote $P := \{1,\ldots,p\}$ and $Q:= \{p+1,\ldots,n\}$ we have
$$R_c = \{ \pm \e_i \pm \e_j\ |\ i,j\in P\} \cup \{ \pm \e_i \pm \e_j,\ |\ i,j\in Q\}$$
and therefore, if $P_\s := \s^{-1}(P), Q_\s := \s^{-1}(Q)$ and $k_{P/Q}$ have the same meaning as in the previous subsections, we obtain
$$\frac 12 \d(z) = \sum_{i\in P_\s}(4k_P(i) + 2i -2n)\phi_i \es i +
\sum_{i\in Q_\s}(4k_Q(i) + 2i -2n)\phi_i \es i.$$
If $\l>0$ and $i_P$ is the maximum element in $P_\s$ then $4k_P(i_P) + 2i_P -2n >0$ implies $i_P>n$, a contradiction. \par
If $\l<0$, the maps $P_\s/Q_\s\ni i\mapsto 4k_{P/Q}(i) + 2i -2n$ are negative and strictly increasing. This implies that two elements $i<j$ both in $P_\s$ or both in $Q_\s$ satisfy $j-i>2$, forcing $p=q=1$.\par
If $\l=0$ then $i\equiv n \ (\mbox{{mod}}\ 2)$ for every $i=1,\ldots,n$, a contradiction.

\end{document}